\documentclass[a4paper,12pt]{article}
\usepackage[utf8]{inputenc}

\usepackage[left=2.3cm,right=2.3cm,top=2.5cm,bottom=2.5cm]{geometry}
\usepackage{amsmath,makeidx,amssymb,amscd,amsfonts, amsthm,color,epsfig,
fancyhdr,graphicx,latexsym,mathrsfs,multicol,psfrag,setspace,wrapfig}
\usepackage[normalem]{ulem}
\usepackage[
  bookmarksopen,
  colorlinks,
  linkcolor = blue,
  urlcolor  = red,
  citecolor = red,
  menucolor = blue
]{hyperref}
\def\argmin{\text{argmin}}

\def\F{\mathcal{F}}
\newcommand\R{\mathbb{R}}
\newcommand\Z{\mathbb{Z}}

\newcommand\N{\mathbb{N}}

\def\argmin{\text{argmin}}
\newtheorem{theo}{Theorem}
\newtheorem{lem}{Lemma}
\newtheorem{cor}{Corollary}
\newtheorem{pr}{Proposition}
\newtheorem{df}{Definition}
\newtheorem{rem}{Remark}

\newcommand{\fim }{{\qed}}
\newcommand{\uf}{{\mathcal{U}}}
\newcommand{\ttau}{{\tilde{\tau}}}

\newcommand{\udd}{{\mathcal{U}^\delta}}
\newcommand{\tuf}{{\widetilde{\mathcal{U}}}}
\newcommand{\ta}{{\tilde{a}}}
\newcommand{\taf}{{\widetilde{\mathcal{A}}}}
\newcommand{\oaf}{{\overline{\mathcal{A}}}}
\newcommand{\af}{{\mathcal{A}}}
\newcommand{\Q}{{\mathfrak{Q}}}
\newcommand\X{H^l(0,S,H^{1+\varepsilon}(D))}

\newcommand{\lb}{{L^2(0,S,H^{1+\varepsilon}(D))}}

\newcommand\he{H^{1+\varepsilon}(D)}

\newcommand\Y{L^2(0,S,W^{1,2}_2(D))}
\newcommand\Ya{L^2(0,S,L^2(D))}

\newcommand\ya{W^{1,2}_2(D)}
\begin{document}

\title{\bf Online Local Volatility Calibration by Convex Regularization}

\author{Vinicius V.L. Albani\thanks{IMPA, Estr. D. Castorina
         110, 22460-320 Rio de Janeiro, Brazil, \href{mailto:vvla@impa.br}{\tt
         vvla@impa.br}} \, and \,
	Jorge P. Zubelli\thanks{IMPA, Estr. D. Castorina
         110, 22460-320 Rio de Janeiro, Brazil, \href{mailto:zubelli@impa.br}{\tt
         zubelli@impa.br}}
        }
\date{\today}
\maketitle
\begin{abstract}

We address the inverse problem of local volatility surface calibration from market given option prices. We integrate the ever-increasing flow of option price information into the well-accepted local volatility model of Dupire. This leads to considering both the local volatility surfaces and their corresponding prices as indexed by the observed underlying stock price as time goes by in appropriate function spaces. The resulting parameter to data map is defined in appropriate Bochner-Sobolev spaces. Under this framework, we prove key regularity properties. This enable us to build a calibration technique that combines online methods with convex Tikhonov regularization tools. Such procedure is used to solve the inverse problem of local volatility identification. As a result, we prove convergence rates with respect to noise and a corresponding discrepancy-based choice for the regularization parameter. We conclude by illustrating the theoretical results by means of numerical tests.
\end{abstract}
\noindent {\bf Keywords:} Local Volatility Calibration, Convex Regularization, Online Estimation, Morozov's Principle, Convergence Rates.
\section{Introduction}\label{sec:intro}
A number of interesting problems in nonlinear analysis are motivated by questions from mathematical finance.
Among those problems, the robust identification of the variable diffusion coefficient that appears in Dupire's local volatility model~\cite{dupire,volguide} presents substantial difficulties for its nonlinearity and ill-posedness. In previous works tools from Convex Analysis and Inverse Problem theory have been used to address this problem. See \cite{acpaper} and references therein. 

In this work, we incorporate the fact that as time evolves more data is available for the identification of Dupire's volatility surface. Thus we develop an {\em online} approach to the ill-posed problem of the local volatility surface calibration. Such surface is characterized by a non-negative two-variable function  $\sigma = \sigma(\tau,K)$  of the time to expiration $\tau$ and the strike price $K$. 

In what follows, we consider that the local volatility surfaces are indexed by the observed underlying asset price $S_0$.
The reason for that stems from the fact that if we try to use information of prices observed on different dates, there is no financial or economical reason for the volatility surface to stay exactly the same. Thus, in principle we may have different volatility surfaces, although such change may be small. 

Let us quickly review the standard Black-Scholes setting and Dupire's local volatility model. Recall that an option or derivative is a contract whose value depends on the value of an underlying stock or index. Perhaps the most well known derivative is an European call option, where the holder has the right (but not the obligation) to buy the underlying at time $t = T$ for a strike value $K$. We shall denote the stochastic process defining such underlying $S(t) = S(t,\omega)$, where as usual we assume that it is an adapted stochastic process on a suitable filtered probability space $(\Omega,\mathscr{U},\mathbb{F},\widetilde{\mathbb{P}})$, where $\mathbb{F} = \{\mathbb{F}_t\}_{t \in \R}$ is a filtration \cite{korn}.

It is well known \cite{dupire,volguide,korn} that, by setting the current time as $t=0$, the value $C$ of an European call option with strike $K$ and expiration $T = \tau$ satisfies:

\begin{equation}
\left\{
\begin{array}{rcll}
-\displaystyle\frac{\partial C}{\partial \tau} + 
\frac{1}{2}\sigma^2(\tau,K)K^2\frac{\partial^2 C}{\partial K^2} - 
bK\frac{\partial C}{\partial K} &=& 0 & \tau > 0, ~K \geq 0\\

C(\tau = 0,K) &=& (S_0 - K)^+, & \text{for}~ K>0,\\
\displaystyle\lim_{K\rightarrow +\infty}C(\tau,K) & = & 0,&\text{for }~ \tau > 0,\\
\displaystyle\lim_{K\rightarrow 0^+}C(\tau,K) & = & S_0,&\text{for }~ \tau > 0
\end{array}
\right.
\label{dup1}
\end{equation}
where $b$ is the difference between the continuously compounded interest and dividend rates of the underlying asset. In what follows, we assume that such quantities are constant.
Defining the diffusion parameter $a(\tau,K) = \sigma(\tau,K)^2/2$, Problem (\ref{dup1}) leads to the following  parameter to solution map:
$$
\begin{array}{rcl}
 F : D(F) \subset X &\longrightarrow & Y\\
  a \in D(F) & \longmapsto & F(a) = C \in Y
\end{array}
$$
where $X$ and $Y$ are Hilbert spaces to be properly defined below. $D(F)$ is the domain of the parameter to solution map (not necessarily dense in $X$) and $C = C(a,\tau,K)$ is the solution of Problem (\ref{dup1}) with diffusion parameter $a$.

The inverse problem of local volatility calibration, as it was tackled in previous works \cite{crepey,acthesis,acpaper,eggeng}, consists in given option prices $C$, find an element $\ta$ of $D(F)$ such that $F(\ta) = C$ in the least-square sense below. Indeed, the operator $F$ is compact and weakly closed. Thus, this inverse problem is ill-posed. In \cite{crepey,acthesis,acpaper,eggeng} different aspects of the Tikhonov regularization were analyzed. In our case, it is characterized by the following: Find an element of
$$
\argmin \left\{\|F(a) - C\|^2_Y + \alpha f_{a_0}(a)  \right\} ~~\text{subject to }~ a \in D(F) \subset X,
$$
where $f_{a_0}$ is a weak lower semi-continuous convex coercive functional. The analysis presented in \cite{crepey,acthesis,acpaper,eggeng} was based on an {\em a priori} choice of the regularization parameter with convex regularization tools.

In contrast, in the present work we explore the dependence of the local volatility surface on the observed asset price in order to incorporate different option price surfaces in the same procedure of Tikhonov regularization. More precisely, we consider the map 
$$
\begin{array}{rcl}
\uf: D(\uf) \subset \mathcal{X} & \longmapsto & \mathcal{Y}\\
\af \in D(\uf) & \longmapsto & \uf(\af): S\in [S_{\min},S_{\max}] \mapsto C(S,a(S))
\end{array}
$$ 
where $C(S,\af(S))$ is the solution of (\ref{dup1}) with $S_0 = S$ and $\sigma^2/2 = a(S)$. Moreover, $\af$ maps $S \in [S_{\min},S_{\max}]$ to $a(S) \in D(F)$ in a well-behaved way.

In this context the inverse problem becomes the following: Given a family of option prices $\mathcal{C} \in \mathcal{Y}$, find $\taf \in D(\uf)$ such that $\uf(\taf) = \mathcal{C}$. We shall see that the operator $\uf$ is also compact and weakly closed. Thus, this problem is also ill-posed. The corresponding regularized problem is defined by the following:

Find an element of
$$
\argmin\left\{ \displaystyle\int_{S_{\min}}^{S_{\max}}\|F(a(S)) - C(S)\|^2_YdS + \alpha f_{\af_0}(\af) \right\} ~~\text{ subject to }~ \af \in D(\uf).
$$

The main contributions of the current work are the following:

Firstly, we extend the local volatility calibration problem to local volatility families. This new setting allows incorporating more data into the calibration problem, leading to an online Tikhonov regularization. We prove that the so-called direct problem is well-posed, i.e., the forward operator satisfies key regularity properties. 
This framework generalizes in a nontrivial way the structure used in previous works \cite{crepey,acthesis,acpaper,eggeng} since it requires the introduction of more tools, in particular that of Bochner spaces.

Secondly, in this setting, we develop a convergence analysis in a general context, based on convex regularization tools. See \cite{schervar}. 

Thirdly, we establish a relaxed version of Morozov's discrepancy principle with convergence rates. This allows us to find the regularization parameter appropriately for the present problem. See \cite{anram,moro}.

The article is divided as follows:

\noindent In Section~2, we present the setting of the direct problem. In Section~3, we define properly the forward operator and prove some key regularity properties that are important in the analysis of the inverse problem. This is done in Theorem~\ref{prop22} and Propositions~\ref{prop4}, \ref{prop6}, \ref{prop7} and \ref{prf1}. In Section~4, we tie up the inverse problem with convex Tikhonov regularization under an {\em a priori} choice of the regularization parameter. The convergence of the regularized solutions to the true one, with respect to $\delta\rightarrow 0$, is stated in Theorem~\ref{tc1}. In Section~5 we establish the Morozov discrepancy principle for the present problem with convergence rates. This is done in Theorems~\ref{tma} and \ref{mor:cr}. Illustrative numerical tests are presented in Section~6.
\section{Preliminaries}\label{sec:preliminar}\label{sec:dupsurv}
We start by setting the so-called direct problem. It is based on the pricing of European call options by a generalization of Black-Scholes-Merton model. 

Performing the change 
of variables $y := \text{log}(K/S_0)$ and $\tau : = T$ on the Cauchy problem (\ref{dup1}) and defining 
$u(S_0,\tau,y) : = C(S_0,\tau,S_0\text{e}^y)$ and $a(S_0,\tau,y) := 
\frac{1}{2}\sigma^2(S_0,\tau,S_0\text{e}^y)$, 
it follows that $u(S_0,\tau,y)$ satisfies
\begin{equation}
\left\{
\begin{array}{rcll}
-\displaystyle\frac{\partial u}{\partial \tau} + a(S_0,\tau,y)\left(\frac{\partial^2 u}{\partial y^2}
 - \frac{\partial u}{\partial y}\right) 
+ b\frac{\partial u}{\partial y} &=& 0 & \tau > 0, ~y \in \R\\
u(\tau = 0,y) &=& S_0(1 - \text{e}^y)^+, &\text{for }~ y \in \R,\\
\displaystyle\lim_{y\rightarrow +\infty}u(\tau,y) & = & 0,&\text{for }~ \tau > 0,\\
\displaystyle\lim_{y\rightarrow -\infty}u(\tau,y) & = & S_0,&\text{for }~ \tau > 0.
\end{array}\right.
\label{dup2}
\end{equation}
Note that, $\sigma$ and $a$ are assumed strictly positive and are related by a smooth bijection (since $\sigma>0$). 
Thus, in what follows we shall work only with the local variance $a$ instead of 
volatility $\sigma$. This simplifies the analysis that follows.

Denote by $D:=(0,T)\times \R$ the set where problem \eqref{dup2} is 
defined. From \cite{eggeng} we know that \eqref{dup2} has a 
unique solution in $W^{1,2}_{2,loc}(D)$, the space of functions 
$u : (\tau,y) \in D \mapsto u(\tau,y) \in \mathbb{R}$ such that, it has locally 
squared integrable weak derivatives up to order one in $\tau$ and up to order two 
in $y$.

We now define the set where the diffusion parameter $a$ lives. For fixed $\varepsilon > 0$, take 
scalar constants $a_1,a_2 \in \mathbb{R}$ such that $0 < a_1 \leq a_2 < +\infty$ 
and a fixed function $a_0 \in \he$, with $a_0 < a < a_1$. Define
\begin{equation}
Q:=  \{a \in a_0 + \he : a_1\leq a \leq a_2\}
\label{domopdi}
\end{equation}
Note that $Q$ is weakly closed and has nonempty interior under the standard topology of $\he$. See the first two chapters of \cite{acthesis,acpaper} and references therein.
\section{The Forward Operator}\label{sec:forward}
Since we assume that the local variance surface is dependent on the current price, we have to introduce proper spaces for the analysis of the problem. As it turns out, we have to make use of Bochner integral techniques. See \cite{evanspde,reedsimon1,yosida}. The main reference for this section is \cite{haschele}.

We start with some definitions. Given a time interval, say $[0,\overline{T}]$, the realized prices $S(t)$ vary within $[S_{\min},S_{\max}]$. After reordering $S(t)$ in ascending order, we perform the change of variables $s = S(t)-S_{\min}$, denote $S = S_{\max}-S_{\min}$. Thus $s \in [0,S]$. Hence, for each $s$, we denote $a(s) := a(s,\tau,y)$ the local variance surface correspondent to $s$.

\begin{df}
Given $\af \in \lb$, with $\af : s\mapsto a(s)$ (see \cite{yosida}), we define its Fourier series $\hat{\af} = \{\hat{a}(k)\}_{k \in \Z}$ by
$$
\hat{a}(k) := \displaystyle\frac{1}{2S}\int^S_0 a(s)\exp(-iks\pi/S)ds + 
\displaystyle\frac{1}{2S}\int^0_{-S} a(-s)\exp(-iks\pi/S)ds.
$$
\end{df}
\noindent It is well defined, since $\{s \mapsto a(s)\exp(-iks2\pi/S)\}$ is weakly measurable and $\lb \subset L^1(0,S,\he)$ by the Cauchy-Schwartz inequality.

We now define a class of Bochner-type Sobolev spaces:
\begin{df}
Let $\X$ be the space of $\af \in \lb$, such that
$$
\|\af\|_l := \displaystyle\sum_{k \in \Z} (1+|k|^l)^2\|\hat{a}(k)\|^2_{\he_\mathbb{C}} < \infty,
$$
where $\he_\mathbb{C} = \he \oplus  i\he$ is the complexification of $\he$. Moreover, $\X$ is a Hilbert space with the inner product
$$
\langle \af,\taf\rangle_l := \displaystyle\sum_{k\in\Z}(1+|k|^l)^2\langle a(k),\ta(k)\rangle_{\he_\mathbb{C}}.
$$
\end{df}
\begin{pr}{\cite[Lemma~3.2]{haschele}}
For $l > 1/2$, each $\af \in \X$ has a continuous representative and the map $i_l : \X \hookrightarrow C(0,S,\he)$ is continuous (bounded).
Moreover, we have the estimate
\begin{equation}
 \displaystyle\sup_{s\in[0,S]}\|u(s)\|_{\he} \leq \|\uf\|_{l}\left(2\sum_{k = 0}^{\infty}\frac{1}{(1+k^l)^2}\right)^{1/2}.
\label{estimate1}
\end{equation}
Defining the application $\langle \af , x\rangle_{\he} : = \{s\mapsto \langle a(s), x \rangle\}$ for each $x$ in $\he$ and $\af$ in $\X$, 
it follows that $\langle \af , x\rangle_{\he}$ is an element of $H^l[0,S]$ and the inequality 
$\|\langle \af , x\rangle_{\he}\|_{ H^l[0,S]} \leq \|\af\|_l\|x\|_{\he}$ holds. Moreover, for every $\af,\mathcal{B} \in \lb$, we have the identity
$$
\langle \af, \mathcal{B} \rangle_\lb = \sum_{k \in \Z}\langle \hat{a}(k),\hat{b}(k)\rangle_{\he_\mathbb{C}}.
$$
\label{p1}
\end{pr}
\begin{lem}
Assume that $l > 1/2$. If the sequence $\{\af_n\}_{n\in\N}$ converges weakly to $\taf$ in $\X$, then, the sequence $\{a_k(s)\}_{k\in\N}$ weakly converges to $\ta(s)$ in $\he$ for every $s \in [0,S]$.
\label{lemw}
\end{lem}
\begin{proof}
Take a $\{\af_n\}_{n\in\N}$ and $\taf$ as above. We want to show that, given a weak zero neighborhood $U$ of $\he$, then for a sufficiently large $n$, $a_n(s) - a(s) \in U$ for every $s \in [0,S]$. A weak zero neighborhood $U$ of $\he$ is defined by a set of $\alpha_1,...,\alpha_K \in \he$ and an $\epsilon > 0$ such that $g \in \he$ is an element of $U$ if $\max_{k = 1,...,K}|\langle g , \alpha_n\rangle| < \epsilon$. 

Since the immersion $H^l[0,S] \hookrightarrow C([0,S])$ is compact and $H^l[0,S]$ is reflexive, it follows that each weak zero neighborhood of $H^l[0,S]$ is a zero neighborhood of $C([0,S])$. Furthermore, from Proposition \ref{p1} we know that 
$\langle \af, \alpha \rangle_{\he} \in H^l[0,S]$ 
with its norm bounded by 
$\|\af\|_l\|\alpha\|_{\he}$, for every $n \in \N$ and $\alpha \in \he$. 
Thus, we take the smallest closed ball centered at zero, $B$, which contains 
$\langle \taf,\alpha_k\rangle_{\he}$
  with $k = 1,...,K$ and every $\langle \af_n,\alpha_k\rangle_{\he}$ with $n\in \N$ and $k = 1,...,K$. 
Therefore, choosing $\epsilon > 0$ as above, it is true that for each $k = 1,...,K$, there are 
$f_{k,1}, ...,f_{k,M(k)} \in H^l[0,S]$ 
and 
$\eta_k > 0$, such that $\|f\|_{C([0,S])} < \epsilon$ 
for every $f \in B$ with $\max_{m = 1,...,M(k)}|\langle f,f_{k,m}\rangle|<\eta_k$.
Hence, we define $\mathcal{C}_{k,m} := \alpha_k \otimes f_{k,m} \in \X^*$ and the weak zero neighborhood $A = \cap^K_{k = 1}A_k$ of $\X$ with
$$
A_k := \{\af \in \X ~: ~|\langle \af, \mathcal{C}_{k,m}\rangle|\leq \eta_k, ~m=1,...,M(k) \}.
$$
As $A$ is a weak zero neighborhood of $\X$, it is true that for sufficiently large $n$, 
$\af_n - \taf \in A$, which implies that $a_n(s) - \ta(s) \in U$ for every $s \in [0,S]$, i.e., $\{a_n(s)\}_{n\in \N}$ 
weakly converges to $\ta(s)$ for every $s \in [0,S]$.
\end{proof}

Define the set 
$\Q := \{ \af \in \X : a(s) \in Q ,~\forall s \in [0,S]\}$, 
i.e., each $\af$ in $\Q$ is the map
$\af :  s \in [0,S] \mapsto a(s) \in Q$. Note that $\Q$ is the space of $Q$-valued paths, with $Q$ defined in (\ref{domopdi}).
\begin{pr}
For $l > 1/2$, the set $\Q$ is weakly closed and its interior is nonempty in $\X$.
\label{p2}
\end{pr}
\begin{proof} 
By Lemma \ref{lemw} and the fact that $Q$ is weakly closed it follows that $\Q$ is weakly closed. The interior of $\Q$ is nonempty since the inclusion $\X \hookrightarrow C(0,S,\he)$ 
is continuous and bounded. 
Note that, given $\epsilon > 0$, it follows that $\taf = \{s \mapsto \ta(s)\}$ with 
$\underline{a} + \epsilon \leq \ta(s) \leq \overline{a} + \epsilon$ 
for every $s \in [0,S]$ is in the interior of $\Q$. 
\end{proof}

We stress that, in what follows, we always assume that $l>1/2$, since it is enough to state our results concerning regularity aspects of the forward operator.

We define below the forward operator, that associates each family of local variance surfaces to the corresponding family of option price surfaces, determined by the Cauchy problem~\eqref{dup2}. Thus, for a given $a_0 \in Q$ we define:
$$
\begin{array}{rcl}
\mathcal{U}: \Q &\longrightarrow& \Y,\\
\af	     &	\longmapsto       & \uf(\af) :s \in [0,S] \mapsto F(s,a(s)) \in \ya,
\end{array}
$$
where $[\uf(\af)](s) = F(s,a(s)):=u(s,a(s))-u(s,a_0)$ and $u(s,a)$ is the solution of the Cauchy problem~\eqref{dup2} with local variance $a$. The following results state some regularity properties concerning the forward operator. See \cite{acpaper} and references therein.

\begin{pr}
The operator $F:[0,S]\times Q\longrightarrow \ya$ is continuous and compact. Moreover, it is sequentially weakly continuous and weakly closed.\label{prop21}
\end{pr}

We define below the concept of Frech\'et equi-differentiability for a family of operators.

\begin{df}
We call a family of operators 
$\{\mathcal{F}_s:Q\longrightarrow\ya \left|~ s \in [0,S] \right.\}$
 Frech\'et equi-differentiable, if for all $\tilde{a} \in Q$ and $\epsilon > 0$, there is a $\delta > 0$, such that
$$
\displaystyle\sup_{s \in [0,S]}\|\mathcal{F}_t(\tilde{a}+h) - \mathcal{F}_s(\tilde{a}) - \mathcal{F}^\prime_s(\tilde{a})h\| \leq \epsilon\|h\|,
$$ 
for $\|h\|_{\he}<\delta$ and $\mathcal{F}^\prime_s(\tilde{a})$ the Frech\'et derivative of $\mathcal{F}_s(\cdot)$ at $\tilde{a}$.
\end{df}

Using this concept, we have the following proposition.
\begin{pr}
The family of operators $\{F(s,\cdot) : Q \longrightarrow \ya \left|~s \in [0,S] \right.\}$ is Frech\'et equi-differentiable.
\label{prop4}
\end{pr}
\begin{proof} 
Given $\ta \in Q$ and $\epsilon > 0$, define $w = F(s,\ta+h) - F(s,\ta) - \partial_a F(s,\ta)h$, it is equivalent to
$w = u(s,\ta+h) - u(s,\ta) - \partial_a u(s,\ta)h$. We denote $v :=  u(s,\ta+h) - u(s,\ta)$. Thus, by linearity $w$ satisfies
$$-w_\tau + \ta(w_{yy} - w_y) + bw_y = h(v_{yy}-v_y),$$
with homogeneous boundary condition. Such problem does not depend on $s$, as $\ta$ is independent of $s$. From the proof of Proposition \ref{prop21} (see also \cite{eggeng}), we have
$
\|w\|_{\ya} \leq C\|h\|_{L^2(D)}\|v\|_{\ya}
$. 
By the continuity of the operator $F$, given $\epsilon > 0 $ we can chose $h \in \he$ with $\|h\|_{\he} \leq \delta$, such that $\|v\|_{\ya} \leq \epsilon /C$ and thus the assertion follows. 
\end{proof}

The following theorem is the principal result of this section, since it states some properties that are at the core of the inverse problems analysis \cite{ern,schervar}. For its proof see Appendix \ref{app:results}.
\begin{theo}
The forward operator $\uf: \Q \longrightarrow \Y$ is well defined, continuous and compact. Moreover, it is sequentially weakly continuous and weakly closed. 
\label{prop22}
\end{theo}

The next result states necessary conditions for the convergence analysis. See \cite{ern,schervar}. Its proof is in the Appendix \ref{app:results}.

\begin{pr}
The operator $\uf(\cdot)$ admits a one sided derivative at $\taf \in \Q$ in the direction $\mathcal{H}$, such that $\taf+\mathcal{H} \in \Q$. The derivative $\uf^\prime(\taf)$ satisfies 
$$
\left\|\mathcal{U}^\prime(\taf)\mathcal{H}\right\|_{\Y} \leq c\|\mathcal{H}\|_{\X}.
$$
Moreover, $\uf^\prime(\taf)$ satisfies the Lipschitz condition
$$
\left\|\uf^\prime(\taf) - \uf^\prime(\taf+\mathcal{H})\right\|_{\mathcal{L}\left(\X,\Y\right)} \leq \gamma\|\mathcal{H}\|_{\X}
$$
for all $\taf,\mathcal{H}\in \Q$ such that $\taf,\taf+\mathcal{H} \in \Q$. \label{prop6}
\end{pr}

The following result is a consequence of the compactness of $\uf(\cdot)$.

\begin{pr}
The Frech\'et derivative of the operator $\uf(\cdot)$ is injective and compact.\label{prop7}
\end{pr}
\begin{proof}
Take $\mathcal{H} \in \ker\left(\uf^\prime(\taf)\right)$. Thus, from the proof of Proposition \ref{prop6}, we have
$
h(s)\cdot (u_{yy} - u_y) = 0.
$ 
However, for each $t$, $G = u_{yy}-u_y$ is the solution of
$$
\left\{\begin{array}{ll}
\partial_\tau G = \displaystyle\frac{1}{2}\left(\partial^2_{yy} - \partial_y\right)\left(a(s)G + bG\right)\\
G\displaystyle\left|_{\tau=0} = \delta(y)\right.,
\end{array}\right.
$$
i.e., $G$ is the Green's function of the Cauchy problem above. Thus, $G > 0$ for every $y$,$\tau > 0$ and $s \in [0,S]$. Therefore $h(t) = 0$. Since this holds for every $s \in [0,S]$, then the result follows.
\end{proof}

We now make use of the bounded embedding of the space
$
\Y 
$
 into the space 
$
L^2(0,S,L^2(D)),
$
since it implies that $\uf$ satisfies the same results presented above with $L^2(0,S,L^2(D))$ instead of $\Y$. 
Thus, we characterize the range of $\uf^\prime(\af)$ as a subset of $L^2(0,S,L^2(D))$ and the range of $\uf^\prime(\af)^*$ 
as a subset of $\X$ in order to proceed in Section~4 the convergence analysis.

\begin{pr}
The operator $\mathcal{U}^\prime(\af^\dagger)^*$ has a trivial kernel.
\label{prf1}
\end{pr}
\begin{proof}
For simplicity take $b = 0$. Denote by
$
\mathcal{L} := -\partial_\tau+a(\partial{yy} - \partial_y)
$ 
the parabolic operator of Equation \eqref{dup2} with homogeneous boundary condition and 
$\mathcal{G}_{u_{yy}-u_y}$ the multiplication operator by $u_{yy}-u_y$. Thus, for each 
$s \in [0,S]$, we have $\partial_a u(s,\ta(s)) = \mathcal{L}^{-1}\mathcal{G}_{u_{yy}-u_y}$, 
where $\mathcal{L}^{-1}$ is the left inverse of $\mathcal{L}$ with null boundary conditions. By definition of
$
\uf^\prime(\taf)^*:L^2(0,S,L^2(D)) \rightarrow \X,
$ 
we have,
$$
\left\langle \mathcal{U}^\prime (\taf)\mathcal{H},\mathcal{Z}\right\rangle_{L^2(0,S,L^2(D))} = 
\langle \mathcal{H}, \Phi \rangle_{\X},$$
$\forall ~\mathcal{H} \in \X$ and $\forall ~\mathcal{Z} \in L^2(0,S,L^2(D))$,
with $\Phi = \uf^\prime (\taf)^*\mathcal{Z}$. 
Thus, given any $\mathcal{Z} \in \ker\left(\uf^\prime (\taf)^*\right)$, it follows that
$$
\begin{array}{rcl}
0 &=& \left\langle \uf^\prime (\taf)\mathcal{H},\mathcal{Z}\right\rangle_{L^2(0,S,L^2(D))} 
= \displaystyle\int^S_0\left\langle\mathcal{L}^{-1}\mathcal{G}_{u_{yy}-u_y}h(s),z(s) \right\rangle_{L^2(D)}ds \\
&=& \displaystyle\int^S_0\left\langle \mathcal{G}_{u_{yy}-u_y}h(s),[\mathcal{L}^{-1}]^*z(s)\right \rangle_{L^2\left(D\right)} ds =\displaystyle\int^S_0\left\langle \mathcal{G}_{u_{yy}-u_y}h(s),g(s)\right\rangle_{L^2\left(D\right)}ds, 
\end{array}
$$
where $g$ is a solution of the adjoint equation
$$
g_\tau+(ag)_{yy} + (ag)_y = z
$$
for each $s \in [0,S]$, with homogeneous boundary conditions. Since $z(t) \in L^2(D)$, we have that $g(s) \in \he$ (see \cite{lady}) and $g \in L^2\left(0,S,\he\right)$. 
Since $\mathcal{G}>0$, from the proof of Proposition \ref{prop7} and the fact that $h \in \X$ is arbitrary, it follows that $g = 0$. Therefore $\mathcal{Z} = 0$ almost everywhere in $s \in [0,S]$. It yields that $\ker\left(\mathcal{U}^\prime (a)^*\right) = \{0\}$. \end{proof}

\begin{rem}
From the last proposition it follows that
$$
\ker\{\uf^\prime(\taf)\} = \{0\} \Rightarrow \overline{\mathcal{R}\left\{\left(\uf^\prime(\taf)\right)^*\right\}} = 
\X.
$$
In other words, the range of the adjoint operator of the Frech\'et derivative of the forward operator $\uf$ at $\taf$ is dense in $\X$.
\end{rem}
To finish this section we shall present below the tangential cone condition for $\uf$. It follows almost directly by the above results and Theorem 1.4.2 from \cite{acthesis}. See also \cite{acpaper2}.


\begin{pr}
The map $\uf(\cdot)$ satisfies the local tangential cone condition
\begin{equation}
\left\|\uf(\af) - \uf(\taf) - \uf^\prime(\taf)(\af - \taf)\right\|_{\Y} \leq \gamma \left\|\uf(\af)- \uf(\taf)\right\|_{\Y}
\end{equation}
for all $\af,\taf$ in a ball $B(\af^*,\rho) \subset \Q$ with some $\rho >0$ and $\gamma < 1/2$.
\label{tang}
\end{pr}
As a corollary we have the following result:
\begin{cor}
The operator $\uf$ is injective.
\end{cor}
\section{The Inverse Problem}\label{sec:tikho}
Following the notation of Section~3, we want to define a precise and robust way of relating each family of European option price surfaces to the corresponding family of local volatility surfaces, both parameterized by 
the underlying stock price. We first present an analysis of existence and stability of regularized solutions, then we establish some convergence rates. We also prove Morozov's discrepancy principle for the present problem with the same convergence rates.

The inverse problem of local volatility calibration can be restated as:

\noindent{\it Given a family of European call option price surfaces  
$\tuf = \{s \mapsto \tilde{u}(s)\}$ in the space $\Ya$, find the correspondent family of local variance surfaces 
$\af^\dagger = \{s \mapsto a^\dagger(s)\} \in \Q$, satisfying
\begin{equation}
\tuf = \uf(\af^\dagger).
\label{ip1a}
\end{equation}}
In what follows we assume that for a given data $\tuf$, the inverse problem \eqref{ip1a} has always a unique solution $\af^\dagger$ in $\Q$. Such uniqueness follows by the forward operator being injective.
Note that, $\tuf$ is noiseless, i.e., is known without uncertainties. This is an idealized 
situation, thus, to be more realistic, we assume that we can only observe corrupted data $\udd$,
 satisfying a perturbed version of (\ref{ip1a}),
\begin{equation}
\udd = \tuf + \mathcal{E} = \uf(\af^\dagger)+\mathcal{E} 
\label{pi1}
\end{equation}
where $\mathcal{E} = \{s \mapsto E(s)\}$ compiles all the uncertainties associated to 
this problem and $\tuf$ is the unobservable noiseless data.  We assume further that, 
the norm of $\mathcal{E}$ is bounded by the noise level $\delta
 > 0$. Moreover, for each $s \in [0,S]$, we assume that $\|E(s)\| \leq \delta/S$. 
These hypotheses imply that
\begin{equation}
\|\udd - \tuf\|_{\Ya} \leq \delta ~\text{ and }~ 
\|u^\delta(s) - \tilde{u}(s)\|_{L^2(D)} \leq \delta/S \text{ for every } s \in [0,S].
\label{errorb}
\end{equation}
Proposition~\ref{prop22} gives that $\uf(\cdot)$ is compact, implying that the 
associated inverse problem is ill-posed. It means that such inverse problem cannot be solved directly in a 
stable way. Hence, we must apply regularization techniques. This, roughly speaking, 
relies on stating the original problem under a more robust setting.
More specifically, instead of looking for an $\af^\delta \in \Q$ satisfying 
(\ref{pi1}), we shall search for an $\af^\delta \in \Q$
 minimizing the Tikhonov functional
\begin{equation}
\mathcal{F}^{\uf^\delta}_{\af_0,\alpha}(\af) = \|\mathcal{U}^\delta - 
\mathcal{U}(\af)\|^2_{\Ya} + \alpha f_{\af_0}(\af).
 \label{tik1}
\end{equation}

The functional $f_{\af_0}$ has the goal of stabilizing the inverse problem and allows us to  incorporate {\em a priori} information through $\af_0$. 

We shall see later that, the minimizers of \eqref{tik1} are approximations for the solution of (\ref{ip1a}).

In order to guarantee the existence of stable minimizers for the functional \eqref{tik1}, we assume that $f_{\af_0} :\Q \rightarrow [0,\infty]$ is convex, coercive and weakly lower semi-continuous. A classical reference on convex analysis is \cite{ekte}. Note that, these assumptions are not too restrictive, since they are fulfilled by a large class of functionals on $\X$. A canonical example is 
$$f_{\af_0}(\af) = \|\af - \af_0\|^2_{\X},$$
which is leads us to the classical Tikhonov regularization.

Recall that $\uf$ is weakly continuous and $\Q$ is weakly closed. Combining that with the required properties of $f_{\af_0}$ we can apply \cite[Theorem 3.22]{schervar}, which gives for a fixed $\udd \in \Ya$ the existence of at least one element of $\Q$ minimizing $\mathcal{F}^{\uf^\delta}_{\af_0,\alpha}(\cdot)$, 
the functional defined in \eqref{tik1}.

For the sake of completeness, we present the definition of stability of a minimizer:

\begin{df}[Stability]
If $\taf$ is a minimizer of \eqref{tik1} with data $\uf$, then it is called stable if for every sequence $\{\uf_k\}_{k \in \N} \subset \Y$ 
converging strongly to $\uf$, the sequence $\{\af_k\}_{k\in \N} \subset \Q$ of 
minimizers of $\mathcal{F}^{\uf^k}_{\af_0,\alpha}(\cdot)$ has a subsequence 
converging weakly to $\taf$.
\label{dfstab}
\end{df}

Then, by \cite[Theorem 3.23]{schervar}, it follows that the minimizers of \eqref{tik1} are stable in the sense of Definition~\ref{dfstab}.

\noindent By \cite[Theorem 3.26]{schervar}, when the noise level $\delta$ and the regularization parameter $\alpha = \alpha(\delta)$ vanish, we can find a sequence of minimizers of \eqref{tik1} converging weakly to the solution of (\ref{ip1a}). In other words, the minimizers of (\ref{ip1a}) are indeed approximations of the family of true local volatility surfaces. 
In addition, as one interpretation of this theorem, we can say that the smaller the noise level $\delta$ is, if the regularization parameter $\alpha$ is properly chosen, the less dependent on the regularization functional and the {\em a priori} information the Tikhonov minimizers are.

Making use of convex regularization tools, we provide some convergence rates with respect to the noise level. In order to do that, we need some abstract concepts, as the Bregman distance related to $f_{\af_0}$, $q$-coerciveness and the source condition related to operator $\uf$. Such ideas were also used in \cite{crepey,acthesis,acpaper,eggeng}, but here they are extended to the context of online local volatility calibration. For the definitions of Bregman distance and $q$-coerciveness see Appendix~\ref{app:def}.

In what follows we always assume that (\ref{ip1a}) has a (unique) solution which is an element of the Bregman domain $\mathcal{D}_B(f_{\af_0})$.

Before stating the result about convergence rates, we need the following auxiliary lemma, which introduces the so-called source condition. For a review on Convex Regularization, see \cite[Chapter 3]{schervar}.
\begin{lem}
For every $\xi^\dagger \in \partial f_{\af_0}(\af^\dagger)$, there exists $\omega^\dagger \in \Ya$
 and $\mathscr{E} \in \X$ such that 
 $\xi^\dagger = \left[\uf^\prime(\af^\dagger)\right]^*\omega^\dagger + \mathscr{E}$ holds. Moreover, $\mathscr{E}$ can be chosen such that $\|\mathscr{E}\|_{\X}$ is arbitrarily small.
\label{lemmax}
\end{lem}

Lemma~\ref{lemmax} follows by $\mathcal{R}(\uf^\prime(\af^\dagger)^*)$ being dense in 
$\X$. See Proposition~\ref{prf1} in Section~3.
Observe also that, we identify $\Ya^*$ and $\X^*$ with $\Ya$ and $\X$, respectively, 
since they are Hilbert spaces.

\begin{theo}[Convergence Rates]
Assume that (\ref{ip1a}) has a (unique) solution. Let the map 
$\alpha : (0,\infty) \rightarrow (0,\infty)$ be such that 
$\alpha(\delta) \approx \delta$ as $\delta\searrow 0$. Furthermore, assume that the convex functional $f_{\af_0}(\cdot)$ is also $q$-coercive with constant $\zeta$, with respect to the norm of $\X$. Then under the source condition of Lemma~\ref{lemmax} it follows that
$$
D_{\xi^\dagger}(\af^\delta_\alpha,\af^\dagger) = \mathcal{O}(\delta) ~~~\text{ and }~~~ 
\|\uf(\af^\delta_\alpha) - \udd \| = \mathcal{O}(\delta).
$$
\label{tc1}
\end{theo}
\begin{proof}
Let $\af^\dagger$ and $\af^\delta_\alpha$ denote the solution of (\ref{ip1a}) and the minimizer of \eqref{tik1}, respectively. It follows that, 
$
\|\uf(\af^\delta_\alpha) - \udd\|^2 + \alpha f_{\af_0}(\af^\delta_\alpha) \leq \|\uf(\af^\dagger) - \udd\|^2 + \alpha f_{\af_0}(\af^\dagger) \leq \delta^2 + \alpha f_{\af_0}(\af^\dagger).
$ 

\noindent Since, $D_{\xi^\dagger}(\af^\delta_\alpha,\af^\dagger) = f_{\af_0}(\af^\delta_\alpha) - f_{\af_0}(\af^\dagger) - \langle \xi^\dagger , \af^\delta_\alpha - \af^\dagger\rangle$, it follows by Lemma~\ref{lemmax} and the above estimate that, 
$$
\|\uf(\af^\delta_\alpha) - \udd\|^2 + \alpha D_{\xi^\dagger}(\af^\delta_\alpha,\af^\dagger) \leq  
\delta^2 - \alpha(\langle \omega^\dagger , \uf^\prime(\af^\dagger)(\af^\delta_\alpha - \af^\dagger)\rangle + \langle \mathscr{E} , \af^\delta_\alpha - \af^\dagger\rangle).
$$

\noindent By Proposition~\ref{tang}, it follows that
$
|\langle \omega^\dagger , \uf^\prime(\af^\dagger)(\af^\delta_\alpha - \af^\dagger)\rangle|
\leq (1+\gamma)\|\omega^\dagger\|\|\uf(\af^\delta_\alpha) - \uf(\af^\dagger)\| \leq (1+\gamma)\|\omega^\dagger\|(\delta + \|\uf(\af^\delta_\alpha) - \udd\|).
$ 
Thus,
$
\|\uf(\af^\delta_\alpha) - \udd\|^2 + \alpha D_{\xi^\dagger}(\af^\delta_\alpha,\af^\dagger) \leq \delta^2 + \alpha (1+\gamma)\|\omega^\dagger\|(\delta + \|\uf(\af^\delta_\alpha) - \udd\|) + \alpha \|\mathscr{E}\|\cdot\|\af^\delta_\alpha - \af^\dagger\|.
$

\noindent Since $\|\mathscr{E}\|$ is arbitrarily small, it follows that, $(\zeta - \|\mathscr{E}\|)/\zeta > 0$.
Moreover, since $f_{\af_0}$ is $q$-coercive with constant $\zeta$ we divide the estimates in two cases, when $q=1$ and $q>1$. 
For the case $q = 1$, the above inequalities imply that,
$$
(\|\uf(\af^\delta_\alpha) - \udd\| - \alpha(1+\gamma)\|\omega^\dagger\|/2)^2 + \alpha(1 - 1/\zeta\|\mathscr{E}\|)D_{\xi^\dagger}(\af^\delta_\alpha,\af^\dagger) \leq
(\delta + \alpha(1+\gamma)\|\omega^\dagger\|)^2
$$
Hence, the assertions follow. 
For the case $q > 1$, we denote $\beta_1 = \|\mathscr{E}\|/\zeta$ and we have that,
$$
\beta_1(D_{\xi^\dagger}(\af^\delta_\alpha,\af^\dagger))^{1/q} \leq \displaystyle\frac{\beta_1^q}{q} + \frac{1}{q}D_{\xi^\dagger}(\af^\delta_\alpha,\af^\dagger).
$$
Thus, assuming that $\beta_1 = \mathcal{O}(\delta^{1/q})$, we have the estimate:
\begin{multline}
\left(\|\uf(\af^\delta_\alpha) - \udd\| - \alpha\frac{1+\gamma}{2}\|\omega^\dagger\|\right)^2 + \displaystyle\alpha\frac{q - 1}{q}D_{\xi^\dagger}(\af^\delta_\alpha,\af^\dagger) \leq\\
(\delta + \alpha(1+\gamma)\|\omega^\dagger\|)^2 + \alpha\displaystyle\frac{\beta_1^q}{q},
\end{multline}
and the assertions follow.
\end{proof}

Note that the rates obtained in Theorem~\ref{tc1} state that, in some sense, the distance between the true local variance and the Tikhonov solution is of order $\mathcal{O}(\delta)$. This can be seen as a measure of the reliability of Tikhonov minimizers for this specific example.
\vspace{1.5cc}

\section{Morozov's Principle}\label{sec:morozov}
We now establish a relaxed version of Morozov's discrepancy principle for the specific problem under consideration \cite{moro}. This is one of the most reliable ways of finding the regularization parameter $\alpha$ as a function of the data $\udd$ and the noise level $\delta$. Intuitively, the regularized solution should not fit the data more accurately than the noise level. We remark that this statement does not follow immediately because, the parameter now has to be chosen as a function of the noise level $\delta$ and the data $\udd$. Thus, it is necessary to prove that such functional in fact satisfies the required criteria to achieve the desired convergence rates.

From Equation (\ref{errorb}), it follows that any $\af \in \Q$ satisfying 
\begin{equation}
\|\uf(\af) - \udd\| \leq \delta
\end{equation}
could be an approximate solution for (\ref{ip1a}). 
If $\af^\delta_\alpha$ is a minimizer of \eqref{tik1}, then Morozov's discrepancy principle says that the regularization parameter $\alpha$ should be chosen through the condition 
\begin{equation}
\|\uf(\af^\delta_\alpha) - \udd\| = \delta
\label{m_init}
\end{equation}
whenever it is possible. In other words, the regularized solution should not satisfy the data more accurately than up to the noise level.

Since the identity \eqref{m_init} is restrictive, in what follows we combine two strategies. The first one is the relaxed Morozov's discrepancy principle studied in \cite{anram}. The second one is the sequential discrepancy principle studied in \cite{ahm}.

Note that, in the analysis that follows, we also require that if $f_{\af_0}(\af) = 0$ then $\af = \af_0$.

\begin{df}{\cite{anram}}
Let the noise level $\delta > 0$ and the data $\udd$ be fixed. Define the functionals 
\begin{eqnarray}
L:\af \in \Q &\longmapsto & L(\af) = \|\uf(\af) - \udd\|\in\mathbb{R}_+\cup \{+\infty\},\\
H:\af \in \Q &\longmapsto & H(\af) = f_{\af_0}(\af)\in\mathbb{R}_+\cup \{+\infty\},\\
I: \alpha \in \mathbb{R}_+ &\longmapsto & I(\alpha) = \F^{\uf^\delta}_{\af_0,\alpha}(\af^\delta_\alpha)\in\mathbb{R}_+\cup \{+\infty\}.
\label{func_moro}
\end{eqnarray}
We also define the set containing all minimizers of the functional \eqref{tik1} for each fixed $\alpha \in (0,\infty)$ as
$$
M_\alpha: = \left\{\af^\delta_\alpha \in \Q : L(a^\delta_\alpha) \leq L(\af) ,~\forall \af \in \X \right\}.
$$
Note that we have extended $L(\af)$ to be equal to $\|\uf(\af) - \udd\|$ when $\af \in \Q$ and to be equal to $+\infty$ otherwise.
\end{df}
The first strategy above mentioned is defined as follows:
\begin{df}[Morozov Criteria]
For prescribed $1< \tau_1 \leq \tau_2$, choose $\alpha = \alpha(\delta,\udd)$ such that $\alpha>0$ and
\begin{equation}
\tau_1\delta \leq \|\uf(\af^\delta_\alpha) - \udd \| \leq \tau_2\delta
\label{morozov}
\end{equation}
holds for some $\af^\delta_\alpha$ in $M_\alpha$.
\end{df}

If the first is not possible, then we consider the following:

\begin{df}[Sequential Morozov Criteria]
For prescribed $\ttau>1$, $\alpha_0 > 0$ and $0<q<1$, choose $\alpha_n = q^n\alpha_0$ such that the discrepancy
\begin{equation}
\|\uf(\af^\delta_{\alpha_{n}}) - \udd \| \leq \ttau\delta < \|\uf(\af^\delta_{\alpha_{n-1}}) - \udd \|
\label{seqmorozov}
\end{equation}
is satisfied for some $n \in \N$ and some $\af^\delta_{\alpha_{n}} \in M_{\alpha_n}$ and $\af^\delta_{\alpha_{n-1}} \in M_{\alpha_{n-1}}$.
\end{df}


It follows by \cite[Lemma 2.6.1]{tikar} that the functional $H(\cdot)$ is non-increasing and the functionals $L(\cdot)$ and $I(\cdot)$ are non-decreasing with respect to $\alpha \in (0,\infty)$ in the following sense, if $0 < \alpha < \beta$ then we have
$$
\sup_{\af^\delta_\alpha \in M_\alpha}L(\af^\delta_\alpha) \leq \inf_{\af^\delta_\beta \in M_\beta}L(\af^\delta_\beta), \inf_{\af^\delta_\alpha \in M_\alpha}H(\af^\delta_\alpha) \geq \sup_{\af^\delta_\beta \in M_\beta}H(\af^\delta_\beta) \text{  and  } I(\alpha) \leq I(\beta).
$$

By \cite[Lemma 2.6.3]{tikar}, the functional $I(\cdot)$ is continuous and the sets of discontinuities of $L(\cdot)$ and $H(\cdot)$ are at most countable and coincide. If we denote this set by $M$, then $L(\cdot)$ and $H(\cdot)$ are continuous in $(0,\infty) \backslash M$.

Since the set $M_\alpha$ is weakly closed for each $\alpha > 0$, we have the following:

\begin{lem}
For each $\overline{\alpha}>0$, there exist $\af_1,\af_2 \in M_{\overline{\alpha}}$ such that
$$
L(\af_1) = \displaystyle\inf_{\af \in M_{\overline{\alpha}}}L(\af) ~~~\text{and} ~~~ 
L(\af_2) = \displaystyle\sup_{\af \in M_{\overline{\alpha}}}L(\af).
$$
\end{lem}

\begin{pr}
Let $1<\tau_1 \leq \tau_2$ be fixed. Suppose that $\|\uf(\af_0) - \udd\| > \tau_2\delta$. Then, we can find $\underline{\alpha},\overline{\alpha}>0$,  such that
$$
L(\af_1) < \tau_1\delta \leq \tau_2 \delta < L(\af_2),
$$
where $\af_1 :=  \af^\delta_{\underline{\alpha}}$ and $\af_2 := \af^\delta_{\overline{\alpha}}$.
\label{pr7}
\end{pr}
\begin{proof}
First, let the sequence $\{\alpha_n\}_{n \in \N}$ converge to $0$. Then, we can find a sequence 
$\{\af_n\}_{n \in \N}$ with $\af_n \in M_{\alpha_n}$ for each $n \in \N$. Now, let $\af^\dagger$ be an 
$f_{\af_0}$-minimizing solution of (\ref{pi1}). Hence, it follows that 
$
L(\af_n)^2 \leq I(\alpha_n) \leq \mathcal{F}^{\uf^\delta}_{\af_0,\alpha_n}(\af^\dagger) \leq \delta^2 + \alpha_n f_{\af_0}(\af^\dagger).
$ 
Thus, for a sufficiently large $n\in\N$, $L(\af_n)^2 < (\tau_1\delta)^2$, since $\alpha_n f_{\af_0}(a^\dagger) \rightarrow 0$. Thus, we can set $\underline{\alpha} := \alpha_n$ for this same $n$ .

We now assume that $\alpha_n \rightarrow \infty$. Taking $\af_n$ as before, we have the following estimates 
$
H(\af_n) \leq \displaystyle\frac{1}{\alpha_n}I(\alpha_n) \leq  \displaystyle\frac{1}{\alpha_n}\mathcal{F}^{\uf^\delta}_{\af_0,\alpha_n}(\af_0) =  
\displaystyle\frac{1}{\alpha_n}\|\uf(\af_0) - \udd\| \rightarrow 0$ whenever $n\rightarrow \infty$. 
Thus, $\displaystyle\lim_{n\rightarrow \infty}f_{\af_0}(\af_n) = 0$, which implies that $\{\af_n\}_{n \in \N}$ converges weakly to $\af_0$. 
Then, by the weak continuity of $\uf(\cdot)$ and the lower semi-continuity of the norm, it follows that 
$$
\|\uf(\af_0) - \udd\| \leq \displaystyle\liminf_{n\rightarrow \infty}\|\uf(\af_n) - \udd\|,$$
which shows the existence of $\overline{\alpha}$, such that 
$$L(\af^\delta_{\overline{\alpha}}) > \tau_2\delta.$$
\end{proof}

\begin{rem}
 For prescribed $1<\tau_1\leq \tau_2$, the discrepancy principle \eqref{morozov} always works if we assume that there is no $\alpha > 0$ such that the minimizers $\af_1,\af_2 \in M_{\alpha}$ satisfy
\begin{equation}
\|\uf(\af_1) - \udd\|  < \tau_1\delta \leq \tau_2\delta < \|\uf(\af_2) - \udd\|.
\label{moro_condition} 
\end{equation}
In other words, only one of the inequalities of the discrepancy principle \eqref{morozov} could be violated by the minimizers associated to $\alpha$. A sufficient condition for such assumption is the uniqueness of Tikhonov minimizers which we are not able to prove for this specific case. Thus, we have to introduce the sequential discrepancy principle \eqref{seqmorozov} whenever the condition \eqref{moro_condition} is violated. Note that the discrepancy principle \eqref{morozov} is always preferable since its lower inequality implies that the Tikhonov minimizers satisfying \eqref{morozov} do not reproduce noise. Whereas the same conclusion cannot be achieved with the sequential discrepancy principle \eqref{seqmorozov}. 
See also \cite[Remark~4.7]{schu} for another discussion about the discrepancy principle~\eqref{morozov}.
\end{rem}
Under the condition \eqref{moro_condition} and Proposition~\ref{pr7}, by \cite[Theorem 3.10]{anram} we can always find $\alpha := \alpha(\delta)>0$ and a Tikhonov minimizer $\af^\delta_\alpha \in M_\alpha$, such that both the inequalities of the discrepancy principle \eqref{morozov} are satisfied. Proposition~\ref{pr7} also implies that the sequential discrepancy principle \eqref{seqmorozov} is well posed. See \cite[Lemma~2]{ahm}. For a convergence analysis under the sequential Morozov, see \cite{hm}. 

\begin{theo}
Assume that the inverse problem \eqref{ip1a} has a (unique) solution. If condition \eqref{moro_condition} holds, then the regularizing parameter $\alpha = \alpha(\delta,\udd)$ obtained through Morozov's discrepancy principle (\ref{morozov}) satisfies the limits
$$
\displaystyle\lim_{\delta \rightarrow 0+}\alpha(\delta,\udd) = 0
~~~\text{       and       }~~~
\displaystyle\lim_{\delta \rightarrow 0+}\frac{\delta^2}{\alpha(\delta,\udd)} = 0.
$$
The same limits hold if $\alpha$ is chosen through the sequential discrepancy principle \eqref{seqmorozov}.
\label{tma}
\end{theo}
\begin{proof}
Let $\{\delta_n\}_{n\in \N}$ be a sequence such that $\delta_n \downarrow 0$ and let $\tuf$ be the noiseless data. Thus, $\|\tuf - \uf^{\delta_n}\|\leq \delta_n$. In addition, recall that the inverse problem \eqref{ip1a} has a unique solution $\af^\dagger$ and then $\uf(\af^\dagger) = \tuf$. We only prove the case where the choice of the regularization parameter is based on the discrepancy principle \eqref{morozov}. Very similar arguments to the ones that follow show the theorem's claim when the choice is based on the sequential discrepancy principle \eqref{seqmorozov}. See \cite[Theorem 1]{ahm}. Thus, it is straightforward to build diagonal convergent subsequences with elements satisfying one of both strategies, in order to prove the limits above asserted.

Let $\alpha_n := \alpha(\delta_n,\uf^{\delta_n})$ denote the regularizing parameter chosen through \eqref{morozov}. Thus, we denote by $\af_n:=\af^{\delta_n}_{\alpha_n}$ its associated minimizer of \eqref{tik1} with respect to $\delta_n$, $\alpha_n$ and $\uf^{\delta_n}$. This defines the sequence $\{\af_n\}_{n\in\N}$, which is pre-compact by the coerciveness of $f_{\af_0}$. Choose a convergent subsequence, denoting it by $\{\af_k\}_{k\in\N}$ and its weak limit by $\taf$. We shall see that $\taf = \af^\dagger$ and thus the original sequence is bounded and has the unique cluster point $\af^\dagger$.

The weakly lower semi-continuity of $\|\uf(\cdot)-\tuf\|$ and $f_{\af_0}$ implies that $\|\uf(\taf) - \tuf\| \leq \lim_{k\rightarrow\infty}(\tau_2+1)\delta_k = 0$. Thus, $\taf$ is a solution of the inverse problem \eqref{ip1a}, which is unique, then $\taf = \af^\dagger$. 

Since, for each $k$, $\af_k$ is a Tikhonov minimizer satisfying the discrepancy principle \eqref{morozov}, it follows by the weakly lower semi-continuity of $f_{\af_0}$ that
\begin{equation}
f_{\af_0}(\af^\dagger) \leq \displaystyle\liminf_{k\rightarrow\infty}f_{\af_0}(\af_k) \leq 
\displaystyle\limsup_{k\rightarrow\infty}f_{\af_0}(\af_k) \leq f_{\af_0}(\af^\dagger).
\label{moro4}
\end{equation}
In other words, $f_{\af_0}(\af_k)\rightarrow f_{\af_0}(\af^\dagger)$.

We now prove that $\alpha(\delta,\uf^\delta)\rightarrow 0$. Assume that with respect to the sequence of the beginning of the proof, there exist $\overline{\alpha}>0$ and a subsequence $\{\alpha_k\}_{k\in\N}$ such that $\alpha_k \geq \overline{\alpha}$ for every $k \in \N$. Denote also by $\{\af_k\}_{k\in\N}$ a sequence of minimizers of \eqref{tik1} with respect to $\delta_k$, $\alpha_k$ and $\uf^{\delta_k}$. Define further the sequence $\{\oaf_k\}_{k\in\N}$ of minimizers of \eqref{tik1} with respect to $\delta_k$, $\overline{\alpha}$ and $\uf^{\delta_k}$. 
Since $L$ in non-decreasing, by the discrepancy principle \eqref{morozov},
\begin{equation}
\|\uf(\oaf_k) - \uf^{\delta_k}\| \leq \|\uf(\af_k) - \uf^{\delta_k}\| \leq \tau_2\delta_k\rightarrow 0
\label{moro5}
\end{equation}
On the other hand, $\displaystyle\limsup_{k\rightarrow \infty} \overline{\alpha}f_{\af_0}(\oaf_k) \leq \overline{\alpha}f_{\af_0}(\af^\dagger)$. By the coerciveness of $f_{\af_0}$, the sequence has a convergent subsequence, denoted also by $\{\oaf_k\}_{k \in \N}$, with limit $\oaf \in \Q$. Thus, by the estimates (\ref{moro4}) and (\ref{moro5}), the weakly lower semi-continuity of $\|\uf(\cdot) - \tuf\|$ and $f_{\af_0}$, it follows that $\|\uf(\oaf) - \tuf\| =0$ and $f_{\af_0}(\oaf) \leq f_{\af_0}(\af^\dagger)$. 
Since the inverse problem \eqref{ip1a} has a unique solution, $\oaf = \af^\dagger$ and thus 
$ f_{\af_0}(\oaf_k)\rightarrow  f_{\af_0}(\af^\dagger)$. On the other hand, $\oaf$ is a minimizer of \eqref{tik1} with regularization parameter $\overline{\alpha}$ and the noiseless data $\tuf$, since for each $\af \in \Q$, the following estimate hold:
$$
\begin{array}{rcl}
\|\uf(\oaf) - \tuf\|^2 + \overline{\alpha} f_{\af_0}(\oaf)& \leq &
\displaystyle\liminf_{k\rightarrow \infty}\left(\|\uf(\af) - \uf^{\delta_k}\|^2 + \overline{\alpha}f_{\af_0}(\af)\right)\\
& = & \|\uf(\af) - \uf^{\delta_k}\|^2 + \overline{\alpha}f_{\af_0}(\af).
\end{array}
$$
Since $f_{\af_0}$ is convex, it follows that for every $t \in [0,1)$
$$
f_{\af_0}((1-t)\oaf + t\af_0) \leq (1-t)f_{\af_0}(\oaf) + tf_{\af_0}(\af_0) = (1-t)f_{\af_0}(\oaf).
$$
Thus, 
$\overline{\alpha}f_{\af_0}(\oaf) \leq \|\uf((1-t)\oaf + t\af_0) - \tuf\|^2 + \overline{\alpha}(1-t)f_{\af_0}(\oaf)$. This implies that 
$\overline{\alpha}tf_{\af_0}(\oaf) \leq \|\uf((1-t)\oaf + t\af_0) - \tuf\|^2$. Since $\tuf = \uf(\oaf)$, by Proposition~\ref{prop6} with $\mathcal{H} = \af_0 - \af$, 
$\overline{\alpha}f_{\af_0}(\oaf) \leq \displaystyle\lim_{t\rightarrow 0^+}\frac{1}{t}\|\uf((1-t)\oaf + t\af_0) - \tuf\|^2 = 0$. Therefore, $f_{\af_0}(\oaf) = 0$. 
But, by hypothesis, it could only hold if $\oaf = \af_0$, i.e., $\af^\dagger = \af_0$. However, $\|\uf(\af_0) - \uf^\delta\| \geq \tau_2\delta$. 
This is a contradiction. We conclude that $\alpha(\delta,\uf^\delta)\rightarrow 0$ when $\delta \rightarrow 0$.

In order to prove the second limit, consider again the subsequence $\{\af_k\}_{k\in\N}$ converging weakly to $\af^\dagger$, the solution of the inverse problem \eqref{ip1a}, when $\delta_k\downarrow 0$. Thus, since for each $k$ $\af_k$ satisfies the discrepancy principle (\ref{morozov}), it follows that $\tau^2_1\delta^2_k + \alpha_k f_{\af_0}(\af_k) \leq \delta_k^2 + \alpha_k f_{\af_0}(\af^\dagger)$. This implies that 
$(\tau_1^2 - 1)\displaystyle\frac{\delta_k^2}{\alpha_k} \leq f_{\af_0}(\af^\dagger) - f_{\af_0}(\af_k) \rightarrow 0$.
\end{proof}

The following theorem states that, if the regularization parameter $\alpha$ is chosen through the discrepancy principle \eqref{morozov}, we achieve the same convergence rates of the Theorem~\ref{tc1}.

\begin{theo}
Assume that the inverse problem \eqref{ip1a} has a (unique) solution.
Suppose that $\af^\delta_\alpha$ is a minimizer of \eqref{tik1} and $\alpha = \alpha(\delta,\udd)$ is chosen through the discrepancy principle (\ref{morozov}) or the sequential discrepancy principle \eqref{seqmorozov}. Then, by the source condition of Lemma~\ref{lemmax}, we have the estimates
\begin{equation}
\|\uf(\af^\delta_\alpha) - \uf(\af^\dagger)\| = \mathcal{O}(\delta)   ~~~\text{ and }~~~
D_{\xi^\dagger}(\af^\delta_\alpha,\af^\dagger) = \mathcal{O}(\delta),
\label{rates_conv}
\end{equation}
with $\xi^\dagger \in \partial f_{\af_0}(\af^\dagger)$. The estimates are achieved whenever \eqref{morozov} is used. 
 \label{mor:cr}
\end{theo}

\begin{proof}
Let $\af^\dagger$ be the solution of the inverse problem (\ref{ip1a}). If $\af^\delta_\alpha \in M_\alpha$, then, the first estimate is trivial since $
\|\uf(\af^\delta_\alpha) - \uf(\af^\dagger)\| \leq (\tau_2 + 1)\delta.$

\noindent If condition \eqref{moro_condition} holds, then by the first inequality of the discrepancy principle (\ref{morozov}), $\tau_1\delta^2 +\alpha f_{\af_0}(\af^\delta_\alpha) \leq \delta^2+\alpha f_{\af_0}(\af^\dagger)$, implying that $f_{\af_0}(\af^\delta_\alpha) \leq f_{\af_0}(\af^\dagger)$, since $\tau_1 - 1 > 0$. 
Hence, for every $\xi^\dagger \in \partial f_{\af_0}(\af^\dagger)$ satisfying the source condition of Lemma~\ref{lemmax} and assuming that $f_{\af_0}$ is $1$-coerciveness with constant $\zeta$, we have the estimates:
\begin{equation}
\begin{array}{rcl}
D_{\xi^\dagger}(\af^\delta_\alpha,\af^\dagger) &\leq& 
|\langle \xi^\dagger , \af^\delta_\alpha - \af^\dagger \rangle| = 
|\langle \uf^\prime(\af^\dagger)^*\omega^\dagger + \mathcal{E}, \af^\delta_\alpha - \af^\dagger\rangle|\\
&\leq& 
\|\omega^\dagger\|\|\uf^\prime(\af^\dagger)(\af^\delta_\alpha - \af^\dagger)\| + 
\|\mathcal{E}\|\|\af^\delta_\alpha - \af^\dagger\|\\
&\leq&
(1+\gamma)\|\omega^\dagger\|\|\uf(\af^\delta_\alpha) - \uf(\af^\dagger)\| + 
\displaystyle\frac{1}{\zeta}\|\mathcal{E}\|D_{\xi^\dagger}(\af^\delta_\alpha,\af^\dagger)
\end{array}
\label{mo:eqcr}
\end{equation}
Since $\xi^\dagger$ can be chosen with $\|\mathcal{E}\|$ arbitrarily small, it follows that
$
1 - 1/\zeta\|\mathcal{E}\| > 0
$ 
and then, by \eqref{morozov},
$$
D_{\xi^\dagger}(\af^\delta_\alpha,\af^\dagger) \leq \displaystyle\frac{\zeta}{\zeta - \|\mathcal{E}\|}
(1+\gamma)\|\omega^\dagger\|\|\uf(\af^\delta_\alpha) - \uf(\af^\dagger)\| \leq 
\tau_2 \frac{\zeta}{\zeta - \|\mathcal{E}\|}(1+\gamma)\|\omega^\dagger\|\cdot \delta.
$$
On the other hand, let $\alpha$ be given by the sequential discrepancy principle \eqref{seqmorozov}. Since,
$ \alpha D_{\xi^\dagger}(\af^\delta_\alpha,\af^\dagger) \leq \|\uf(\af^\delta_\alpha)-\udd\|^2 + \alpha D_{\xi^\dagger}(\af^\delta_\alpha,\af^\dagger)$, then,
$$
D_{\xi^\dagger}(\af^\delta_\alpha,\af^\dagger) \leq \displaystyle\frac{\delta^2}{\alpha} + |\langle \xi^\dagger , \af^\delta_\alpha - \af^\dagger \rangle|.
$$
By the previous case, the second term in the right hand side of the above inequality has the order $\mathcal{O}(\delta)$. By Theorem~\ref{tma} the first term also vanishes. Since $\ttau\delta \leq  \|\uf(\af^{\delta}_{\alpha/q}) - \udd\|$, it follows that the first term is of order $\mathcal{O}\left(|f_{\af_0}(\af^\delta_{\alpha/q}) - f_{\af_0}(\af^\dagger)|\right)$ and $|f_{\af_0}(\af^\delta_{\alpha/q}) - f_{\af_0}(\af^\dagger)| \leq |\langle \xi^\dagger , \af^\delta_{\alpha/q} - \af^\dagger \rangle|$. See \cite[Proposition~10]{ahm}.
\end{proof}

As above mentioned, the above rates obtained in terms of Bregman distance state that, in some sense, the distance between the true local variance and the Tikhonov solution is of order $\mathcal{O}(\delta)$. Under a more practical perspective, consider $f_{\af_0}(\af) = \|\af - \af_0\|^2_{\X}$. In this case, it follows that  $\|\af^\delta_\alpha - \af^\dagger\|_{\X} = \mathcal{O}(\delta^{1/2})$. In addition, if $l > 1/2$ in $\X$, it follows by the inequality \eqref{estimate1} that
$$\sup_{s\in [0,S]}\|a^\delta_\alpha(s) - a^\dagger(s)\|_{\he} \leq C \|\af^\delta_\alpha - \af^\dagger\|_{\X}.$$
Thus, the convergence rates also follows uniformly in $s$ and imply the convergence rates obtained in previous works, such as \cite{acpaper, eggeng, crepey}. This can be understood as the online solution is at least as good as the solution obtained in the standard case, i.e., the Tikhonov minimizers with only one price surface.

\begin{rem}
For $f_{\af_0}$ $q$-coercive with $q > 1$, a reasoning as the one used in Equation~\eqref{mo:eqcr}, gives that  
$$\begin{array}{rcl}
D_{\xi^\dagger}(\af^\delta_\alpha,\af^\dagger) 
&\leq& \beta_1 (D_{\xi^\dagger}(\af^\delta_\alpha,\af^\dagger))^{1/q} + \beta_2\|\uf(\af^\delta_\alpha) - \uf(\af^\dagger)\|\\
&\leq& \displaystyle\frac{\beta_1^q}{q} + \frac{1}{q}D_{\xi^\dagger}(\af^\delta_\alpha,\alpha) + \beta_2\|\uf(\af^\delta_\alpha) - \uf(\af^\dagger)\|.
\end{array}$$
Assume further that $\beta_1 = \mathcal{O}(\delta^{\frac{1}{q}})$. Since 
$
\|\uf(\af^\delta_\alpha) - \uf(\af^\dagger)\| = \mathcal{O}(\delta)
$, 
it follows that 
$
\|\af^\delta_\alpha - \af^\dagger\|^q \leq \displaystyle\frac{1}{\zeta}D_{\xi}(\af^\delta_\alpha,\af^\dagger)  = \mathcal{O}(\delta).
$
\end{rem}

\section{Numerical Results}\label{sec:numerics}
We first perform tests with synthetic data for testing accuracy and advantages of the methods. Then, we present some examples with observed market prices.

We note that Problem~\eqref{dup2} is solved by a Crank-Nicolson scheme \cite[Chapter 5]{vvlathesis}. Since we shall use a gradient-based method to solve numerically the minimization of the Tikhonov functional \eqref{tik1}. Let $J^\delta(\af)$ and $\nabla J^\delta(\af)$ denote the quadratic residual and its gradient, respectively. More precisely, the residual is given by 
$J^\delta(\af) : =  \|\uf(\af) - \udd\|^2_{\Ya} = \int^S_0\|F(s,a(s)) - u^\delta(s) \|^2_{L^{2}(D)}ds$ and the gradient is given by
\begin{multline}
\langle \nabla J^\delta(\af),\mathcal{H}\rangle_{\X}  =  2\langle\uf(\af) - \udd ,\uf^\prime(\af)\mathcal{H}\rangle_{\Ya}\\
  =  2\displaystyle\int^S_0\int_D\{[v(u_{yy}-u_y)h(t)](s,a(s))\}(\tau,y)d\tau dyds,
\label{gradj}
\end{multline}
where, for each $s \in [0,S]$, $v$ is the solution of equation,
\begin{equation}
v_\tau + (av)_{yy} + (av)_y +bv_y= u(t,a) - u^\delta(s)
\label{adj}
\end{equation}
with homogeneous boundary condition. Note that, $V = \{V: s \mapsto v(s)\}$ is an element of $\Y$. We also numerically solve Problem (\ref{adj}) by a Crank-Nicolson scheme. See \cite[Chapter 5]{vvlathesis}.

In the following examples we assume that $l=1$ in $\X$ and the regularization functional is
$
f_{\af_0}(\af) = \displaystyle\|\af - \af_0\|^2_{\X}.
$

\subsection{Examples with Synthetic Data}
Consider the following local volatility surface:
$$
a(s,u,x) = \left\{
\begin{array}{ll}
\label{sig}
\displaystyle\frac{2}{5}\left(1 - \frac{2}{5}\text{e}^{-\frac{1}{2}( u - s)} \right)\cos(1.25\,\pi \,x),&(u,x) \in (0,1]\times \left[-\displaystyle\frac{2}{5},\displaystyle\frac{2}{5}\right],\\

\displaystyle\frac{2}{5}, & \text{otherwise.} 
\end{array}
\right.
$$

We generate the data, i.e., evaluate the call prices with the above volatility, in a very fine mesh. Then we add a zero-mean Gaussian noise with standard deviation $\delta = 0.035,\, 0.01$. We interpolate the resulting prices in coarser grids. This avoids a so-called inverse crime \cite{somersalo}.

In the present test, we assume that, $r = 0.03$, $(\tau,y) \in [0,1]\times [-5,5]$. We generate the price data with step sizes $\Delta \tau = 0.002$ and $\Delta y = 0.01$. Then, we solve the inverse problem with the step sizes $\Delta \tau = 0.01, \,0.005$ and $\Delta y = 0.1$.  We also assume that the asset price is given by $s \in [29.5, 32.5]$ with three different step sizes, $\Delta s = 0.25, 0.1, 0.01$.

In what follows, we refer to standard Tikhonov as the case when we consider a single price surface in the Tikhonov regularization. Whereas, we use the terminology ``online'' Tikhonov whenever we use more than one single price surface.

\begin{figure}[ht]
  \begin{center}
    \begin{tabular}{rcl}
      \includegraphics[width=0.30\textwidth]{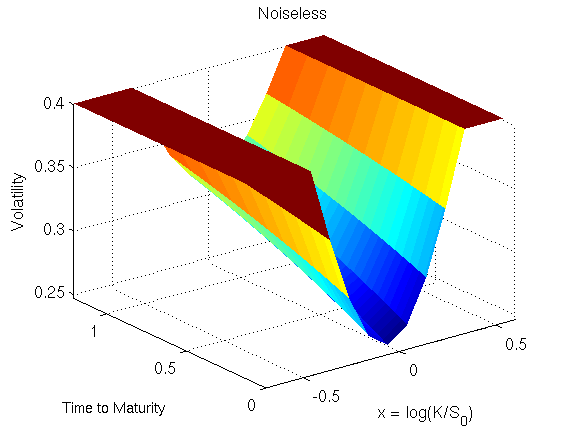}&
      \includegraphics[width=0.30\textwidth]{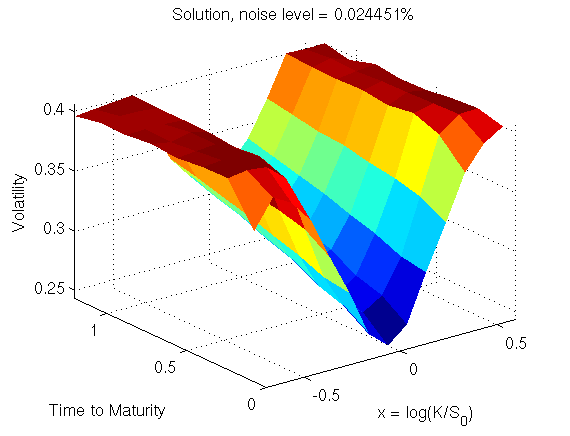}&
      \includegraphics[width=0.30\textwidth]{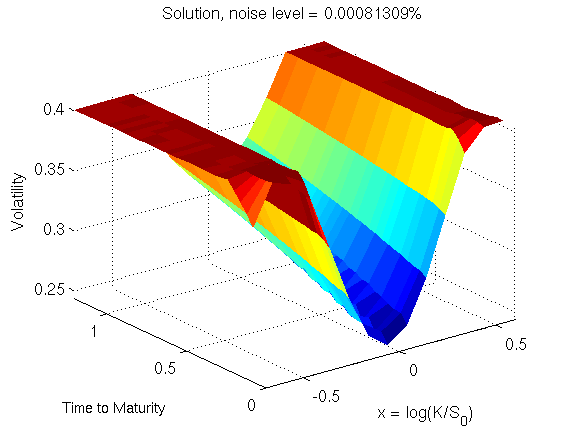}
    \end{tabular}
  \end{center}
  \caption{Left: Original local volatility. Center: Reconstruction with noise level $\delta = 0.035$. Right: Reconstruction with $\delta = 0.01$. When the noise level decreases, the reconstructions become more accurate.}
  \label{test1}
\end{figure}
Figure~\ref{test1} shows reconstructions of the local volatility surface from price data with different noise levels. In addition, we can see that, when the noise level decreases, by refining the accuracy of the data, the resulting reconstructions become more similar to the original local volatility surface. This is an illustration of the Theorems \ref{tc1}, \ref{tma} and \ref{mor:cr}.

\begin{figure}[ht]
  \begin{center}
    \begin{tabular}{rl}
      \includegraphics[width=0.46\textwidth]{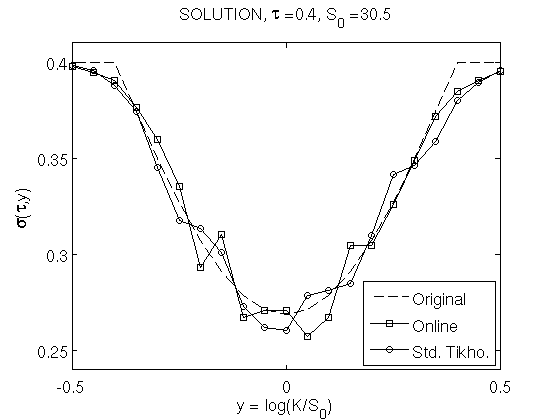}&
      \includegraphics[width=0.46\textwidth]{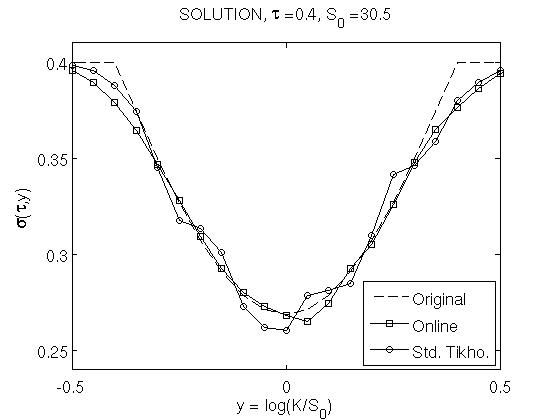}
    \end{tabular}
  \end{center}
  \caption{Comparison between standard and online Tikhonov. As the number of price surfaces increases, the reconstructions become more accurate.}
  \label{test2}
\end{figure}
In Figure~\ref{test2}, we can see that the online Tikhonov presents better solutions than the standard one, as we increase the number of price surfaces in the calibration procedure. Here, the regularization parameter was obtained through the Morozov's discrepancy principle.

\begin{figure}[ht]
  \begin{center}
      \includegraphics[width=0.47\textwidth]{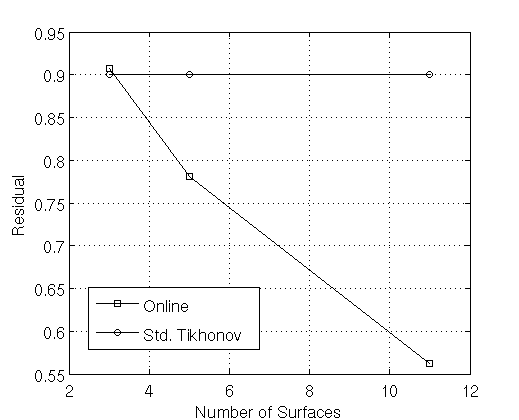}
  \end{center}
  \caption{$L^2$ distance between original local variance and its reconstructions, as a function of the number of price surfaces. it is constant for standard Tikhonov and non-increasing for on line Tikhonov.}
  \label{test3}
\end{figure}
Figure~\ref{test3} shows the evolution of the $L^2(D)$ distance between the reconstructions and the original local variance as a function of the number of surfaces of call prices: it is constant for standard Tikhonov and non-increasing for online Tikhonov. 

\subsection{Examples with Market Data}

We now present some reconstructions of the local volatility by online Tikhonov regularization from market prices. We solve the inverse problem with the step sizes $\Delta \tau = 0.01$ and $\Delta y = 0.1$. The regularizing functional is $f_{\af_0}(\af) = \|\af - \af_0\|^2_{\X}$ and the regularization parameter is chosen through the discrepancy principle \eqref{morozov}. We estimate the noise level as half of the mean of the bid-ask spread in market prices. The market prices are interpolated linearly in the mesh where the inverse problem is solved. In the present example, we consider seven surfaces of call prices in each experiment. The data corresponds to vanilla option prices on futures of Light Sweet Crude Oil (WTI) and Henry Hub natural gas. For a survey on commodity markets, see the book \cite{geman}. For a study of of an application of Dupire's local volatility model on commodity markets, see \cite[Chapter~4]{vvlathesis}.

\begin{figure}[ht]
  \begin{center}
    \begin{tabular}{cc}
      \includegraphics[width=0.46\textwidth]{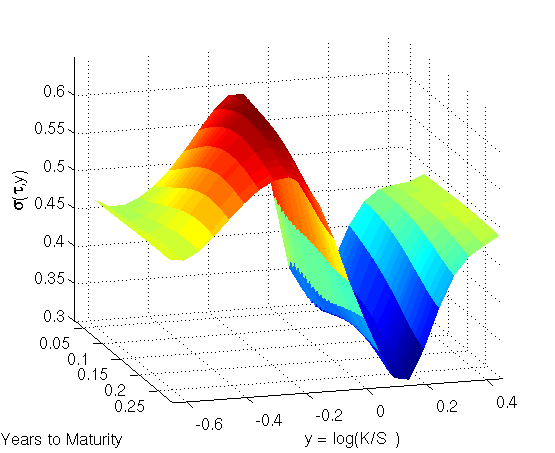}&
      \includegraphics[width=0.46\textwidth]{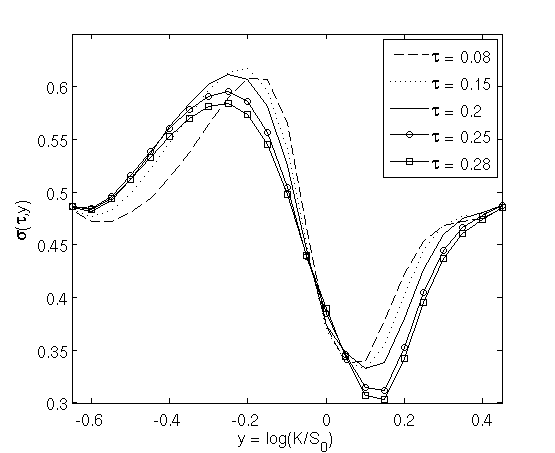}
    \end{tabular}
  \end{center}
  \caption{Local Volatility reconstruction from European vanilla options on futures of WTI oil. We used online Tikhonov regularization with the standard quadratic functional.}
  \label{test4}
\end{figure}

Note that, in order to use the framework developed in the previous sections, we assumed that, the local volatility is indexed by the unobservable spot price, instead of the future price. For more details on such examples, see Chapters 4 and 5 of \cite{vvlathesis}.

\begin{figure}[ht]
  \begin{center}
    \begin{tabular}{cc}
            \includegraphics[width=0.46\textwidth]{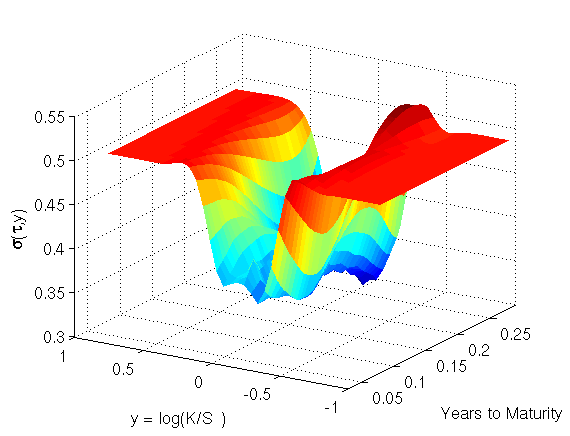}&
      \includegraphics[width=0.46\textwidth]{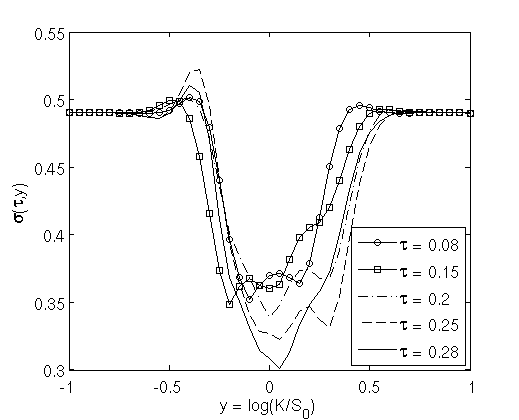}
    \end{tabular}
  \end{center}
  \caption{Local Volatility reconstruction from European vanilla options on futures of Henry Hub natural gas. We used online Tikhonov regularization with the standard quadratic functional.}
  \label{test5}
\end{figure}

Figures~\ref{test4} and~\ref{test5} present the best reconstructions of local volatility for WTI and HH data, respectively. We collected the data prices for Henry Hub natural gas and WTI oil between 2011/11/16 and 2011/11/25, i.e., seven consecutive commercial days.

%
\section{Conclusions}\label{sec:conclusion}

In this paper we have used convex regularization tools to solve the inverse problem associated to  Dupire's local
volatility model when there is a steady flow of data. We first established results concerning existence, stability and convergence of the regularized solutions, making use of convex regularization tools and the regularity of the forward operator. We also proved some convergence rates. Furthermore, we established discrepancy-based choices of the regularization parameter,  under a general framework, following \cite{anram,ahm}. Such analysis allowed us to implement the algorithms and perform numerical tests.

The main contribution, {\em vis a vis} previous works, and in particular~\cite{acpaper}, is that we extended the convex regularization techniques to incorporate the information and data stream that is constantly supplied by the market. Furthermore, we have proved discrepancy-based choices for the regularization parameter that are suitable to this context with regularizing properties.

A natural extension of the current work is the application of these techniques to the context of future markets, where the underlying asset is the future price of some financial instrument or commodity. In such markets, vanilla options represent a key instrument in hedging strategies of companies and in general they are far more liquid than in equity markets. The warning here is that, in general, we do not have an entire price surface. Actually in this case, we only have an option price curve for each future's maturity. Thus, in order to apply the techniques above to this context, it is necessary to assembly all option prices for futures on the same instrument (financial or commodity) in a unique surface in an appropriate way. This was discussed in \cite[Chapter 4]{vvlathesis} and will be published elsewhere.

\section{Acknowledgments}

V.A. acknowledges and thanks CNPq, Petroleo Brasileiro S.A. and Ag\^encia Nacional do Petr\'oleo for the financial support during the preparation of this work.
 J.P.Z. acknowledges and thanks the financial support from CNPq through grants 302161/2003-1 and
474085/2003-1, and from FAPERJ through the programs {\em Cientistas do Nosso Estado} and {\em Pensa Rio}.

\appendix
\section{Proofs, Technical Results and Definitions}
In this appendix we collect technical results and definitions that were used in the remaining parts of the article. We also present the proofs of some results of from Section~3.
\subsection{Bregman Distance and $q$-Coerciveness}\label{app:def}

\begin{df}{\cite[Definition 3.15]{schervar}}
Let $X$ denote a Banach space and $f: D(f) \subset X \rightarrow \R\cup {\infty}$ be a convex functional with sub-differential $\partial f(x)$ in $x \in D(f)$. The Bregman distance (or divergence) of $f$ at $x \in D(f)$ and $\xi \in \partial f(x) \subset X^*$ is defined by
$
 D_{\xi}(\tilde{x},x) = f(\tilde{x}) - f(x) - \langle\xi,\tilde{x} - x\rangle,
$
 for every $\tilde{x} \in X$, with $\langle\cdot,\cdot\cdot\rangle$ the dual product of $X^*$ and $X$. Moreover, the set 
$
\mathcal{D}_B(f) = \{x \in D(f) ~:~ \partial f(x) \not= \emptyset\}
$ 
is called the Bregman domain of $f$.
\end{df}
We stress that the Bregman domain $\mathcal{D}_B(f)$ is dense in $D(f)$ and the interior of $D(f)$ is a subset of $\mathcal{D}_B(f)$. The map $\tilde{x}\mapsto D_{\xi}(\tilde{x},x)$ is convex, non-negative and satisfies $D_{\xi}(x,x) = 0$. In addition, if $f$ is strictly convex, then $D_{\xi}(\tilde{x},x) = 0$ if and only if $\tilde{x} = x$. For a survey in Bregman distances see \cite[Chapter I]{butiusem}.

\begin{df}
 For $1\leq q <\infty$ and $x \in D(f)$, the Bregman distance $D_{\xi}(\cdot,x)$ is said to be $q$-coercive with constant $\zeta>0$ if
$
D_{\xi}(y,x) \geq \zeta \|y-x\|^q_X
$ 
for every $y \in D(f)$.
\end{df}
\subsection{Equicontinuity}
Let $X$ and $Y$ be locally convex spaces. Fix the sets $B_X \subset X$ and $M \subset C(B_X,Y)$. A set $M$ is called equicontinuous on $B_X$ if for every $x_0 \in B_X$ and every zero neighborhood, $V \subset Y$ there is a zero neighborhood $U \subset X$ such that $G(x_0) - G(x) \in V$ for all $G \in M$ and all $x \in B_X$ with $x-x_0 \in U$. Furthermore, $M$ is called uniformly equicontinuous if for every zero neighborhood $V \subset Y$ there exists a zero neighborhood $U \subset X$ such that $G(x) - G(x^\prime) \in V$ for all $G \in M$ and all $x,x^\prime \in B_X$ with $x-x^\prime \in U$.

From \cite{haschele} we have the technical result:
\begin{pr}
Let $F: [0,T]\times B_X \longrightarrow Y$ be a function, and $B_X$, $X$ and $Y$ be as above. If $M_1:= \{F(t,\cdot) : t \in [0,T]\} \subset C(B_X,Y)$, $M_2:=\{F(\cdot,x) : x \in B_X\} \subset C([0,T],Y)$ and $M_1$ (respectively $M_2$) is equicontinuous, then $F$ is continuous. Reciprocally, if $F$ is continuous, then $M_1$ is equicontinuous and if additionally $B_X$ is compact, then $M_2$ is equicontinuous, too.\label{prop11}
\end{pr}
\subsection{Proof of Results from Section~3}\label{app:results}
{\bf Proof of Theorem \ref{prop22}:}
{\it Well Posedness:} Take an arbitrary $\taf \in  \Q$, by the continuity of $\taf$ (see Proposition \ref{p1}) and $F$, it follows that $t \mapsto F(s,\ta(s))$ is continuous and then weakly measurable. Therefore,
$s \mapsto \|F(s,a(s))\|_{\ya}$ is bounded, then $\uf(\taf) \in\Y$, which asserts the well-posedness of $\uf(\cdot)$.

\noindent {\it Continuity:} As $F:[0,S]\times Q \longrightarrow \ya$ is continuous, it follows by Proposition \ref{p1} that the set $\{F(s,\cdot) \left|~ s \in [0,S]\right.\} \subset C(Q,\ya)$ is uniformly equicontinuous, i.e., given $\epsilon > 0$, there is a $\delta > 0$ such that, for all $a,\ta \in Q$ satisfying $\|a - \ta\| < \delta$, we have that 
$\sup_{s \in [0,S]}\|F(s,a)-F(s,\ta)\| < \epsilon.$ 
Thus, given $\epsilon > 0$ and $\af, \taf \in \Q$ such that $\sup_{s \in [0,S]}\|a(s) - \ta(s)\|_{\he}<\delta$, then, by the uniform equicontinuity of $\{F(s,\cdot), s \in [0,S]\}$, it follows that
$$
\displaystyle\|\uf(\af) - \uf(\taf)\|^2_{\Y} = \displaystyle\int^S_0\|F(s,a(s)) - F(s,\ta(s))\|^2_{\ya}ds < \epsilon^2\cdot S,
$$
which asserts the continuity of $\uf(\cdot)$.

\noindent {\it Compactness:} It is sufficient to prove that, given an $\epsilon > 0$ and a sequence $\{\af_n\}_{n \in \N}$in $\Q$ converging weakly to $\taf$, it follows that there exist an $n_0$ and a weak zero neighborhood $U$ of $\X$ such that for $n > n_0$, $\af_n-\taf \in U$ and $\|\uf(\af_n) - \uf(\taf)\|_{\Y}< \epsilon.$

Following the same arguments of the proof of Lemma \ref{lemw}, we can find a set of functionals $\mathcal{C}_{n,m} \in \X^*$, defining such zero neighborhood $U$. We first note that, since $F$ is weak continuous, it follows that, given an $\epsilon >0$, there are $\alpha_1,...,\alpha_N \in \he$ and $\delta > 0$, such that 
$\sup_{s \in [0,S]}\|F(s,a) - F(s,\ta)\| < \epsilon/S$  for all $a,\ta \in B$ with 
\begin{equation}
\max\{|\langle a - \ta, \alpha_n \rangle_{\he} |\,:\, n = 1,...,N\} < \delta.\label{p5:eq1}
\end{equation}
By Proposition \ref{p1}, the estimate $\langle \af,\alpha_n\rangle_{\he} \in H^l[0,S]$ holds with its norm bounded by $\|\af\|_l\|\alpha_n\|_{\he}$. Then, there is a closed and bounded ball 
$A \subset H^l[0,S]$ containing $\langle \af,\alpha_n\rangle_{\he}$, for all $n = 1,...,N,$ 
and $\af \in \mathbb{B}$. 

For $n = 1,...,N$ and the same $\delta > 0$ of \eqref{p5:eq1}, there are $f_{n,1},...,f_{n,M(n)}$  in $H^l[0,S]$ and $\xi_n > 0$ such that, $\|f\|_{C([0,S])} < \delta$ for every $f \in A$ satisfying the estimate 
$\max_{m = 1,...,M(n)}|\langle f,\alpha_n\rangle_{\he}| < \xi_n.$ 
Define $\mathcal{C}_{n,m} : = \alpha_n \otimes f_{n,m}$, with $n = 1,...,N$ and $ m=1,...,M(n)$. It is an element of $\X^*$, where, for each $\af \in \X$, we have that 
$\langle\af, \mathcal{C}_{n,m}\rangle_l = \langle \langle \af,\alpha_n\rangle_{\he}, f_{n,m}\rangle_{H^l[0,S]}$ and thus 
$$\langle\af, \mathcal{C}_{n,m}\rangle_l = \displaystyle\sum_{k\in \Z}(1 + |k|^l)^2\langle \hat{a}(k),\alpha_n\rangle_{\he}\hat{f}_{n,m}(k).
$$
These functionals define a weak zero neighborhood $U := \cap^N_{n=1}U_n$ with
$$
U_n : = \{ \af \in \X : |\langle \af, \mathcal{C}_{n,m}\rangle_l| < \xi_n, ~m=1,...,M(n)\}.
$$
Therefore, if $\{\af_k\}_{k\in\N}\subset \mathbb{B}$ converges weakly to $\taf \in \mathbb{B}$, then for a sufficient large $k$, $\af_k-\taf \in U$ and by the definition of $U$, we have that for each $n = 1,...,N$, 
$\xi_n > |\langle \af-\taf, \mathcal{C}_{n,m}\rangle_l| = |\langle \langle \af-\taf,\alpha_n\rangle_{\he}, f_{n,m} \rangle_{H^l[0,S]}|
$ 
for all $m = 1,...,M(n)$.
By the choice of the $f_{n,m} \in H^l[0,S]$, it follows that 
$\|\langle \af_k-\taf,\alpha_n\rangle_{\he}\|_{H^l[0,S]} < \delta$ for all $n = 1,...,N,$ 
which implies that $\|\uf(\af_k) - \uf(\taf)\|_{\Y} \leq \epsilon\cdot T$.

\noindent {\it Weak Continuity:} The weak continuity follows directly from the proof of compactness, as we use the same framework, only changing the compactness of $F$, by the weakly equicontinuity of $\{F(s,\cdot) : ~s \in [0,S]\}$ on bounded subsets of $Q$.

\noindent {\it Weak Closedness:} Just note that the set $\Q$ is weakly closed and the operator $\uf(\cdot)$ is weakly continuous. \fim

{\bf Proof of Proposition \ref{prop6}}
By Proposition \ref{prop4}, the family of operators $\{F(s,\cdot) \,: \,s \in [0,S]\}$ is Frech\'et equi-differentiable. Take $\taf,\mathcal{H} \in \X$, such that $\taf,\taf+\mathcal{H} \in \Q$. Then, define the one sided derivative of $\uf(\cdot)$ at $\taf$ in the direction $\mathcal{H}$ as 
$\uf^\prime(\taf)\mathcal{H} := \{s \mapsto \partial_a F(s,\ta(s))h(s)\}$, 
where for each $s \in [0,S]$, dropping $t$ to easy the notation, $\partial_a F(s,\ta)h$ is the solution of
$$
-v_\tau + a(v_{yy}-v_y) + bv_y = h(u_{yy}-u_y)
$$
with homogeneous boundary conditions and $u = u(s,a(s))$. From Proposition \ref{prop21} we have the estimate
$\|\partial_a F(s,\ta(s))h(s)\|_{\ya} \leq C\|h(s)\|_{L^2(D)}\|u_{yy}(s,\ta(s))-u_{y}(s,\ta(s))\|_{L^2(D)}$. 
Note that, $\|u_{yy}(s,a)-u_{y}(s,a)\|_{L^2(D)}$ is uniformly bounded in $[0,S]\times Q$. Thus, $\uf^\prime(\taf)\mathcal{H}$ is well defined and
\begin{multline}
\left\| \uf^\prime(\taf)\mathcal{H}\right\|^2_{\Y} = \displaystyle\int^S_0\|\partial_a F(s,\ta(s))h(s)\|^2_{\ya}ds  \\ 	 \leq  C \displaystyle\int^S_0\|h(s)\|_{L^2(D)}\|u_{yy}(s,\ta(s))-u_{y}(s,\ta(s))\|_{L^2(D)}ds\\
	 \leq  c\displaystyle\int^S_0\|h(s)\|^2_{L^2(D)}ds = c\|\mathcal{H}\|^2_{\X}
\end{multline}
Therefore, $\mathcal{U}^\prime(\taf)$ can be extended to a bounded linear operator from the space $\X$ into $\Y$.

Let $\taf,\mathcal{H},\mathcal{G} \in \X$ be such that, $\taf,\taf+\mathcal{H},\taf+\mathcal{G}, \taf+\mathcal{H}+\mathcal{G}$ are in $Q$. 
Define $v:=u(s,a(s)+h(s)) - u(s,a(s))$. Thus, 
$$
w := \partial_a u(s,a(s)+h(s))g(s) - \partial_a u(s,a(s))g(s)
$$
satisfies
$$
-w_\tau + a(w_{yy} - w_y) = -g[v_{yy} - v_{y}] - h[(\partial_a u(s,a+h)g)_{yy} - (\partial_a u(s,a+h)g)_{y}],
$$
with homogeneous boundary conditions (dropping the dependence on $s$). As above, we have
\begin{multline}
\left\|\mathcal{U}^\prime(\taf+\mathcal{H})\mathcal{G} - \mathcal{U}^\prime(\taf) \mathcal{G}\right\|^2_{\Y} = \displaystyle\int^S_0\|w\|^2_{\ya}ds\\
	 \leq c_1\displaystyle\int^S_0\|g(s)\|^2_{L^2(D)}\|v_{yy}(s,\ta(s)) - v_y(s,\ta(s))\|^2_{L^{2}(D)}ds\\
	+ c_2 \displaystyle\int^S_0\|h(s)\|^2_{L^2(D)}\|\partial_a u(s,a(s)+h(s))g(s)\|^2_{\ya}ds \\
	\leq C\|\mathcal{H}\|^2_{\X}\|\mathcal{G}\|^2_{\X},
\end{multline}
which yields the Lipschitz condition. \fim
\bibliographystyle{abbrv}

\vspace{1cc}

\noindent Instituto Nacional de Matem\'atica Pura e Aplicada\\
Estr. D. Castorina 110, 22460-320. Rio de Janeiro,\\
Brazil.

\noindent E-mail: \href{mailto:vvla@impa.br}{\tt vvla@impa.br} (Vinicius Albani) and \href{mailto:zubelli@impa.br}{\tt zubelli@impa.br} (Jorge Zubelli).

\end{document}